\documentclass{amsart}
\usepackage{amsfonts, color}
\usepackage{amsfonts}
\usepackage{latexsym}
\usepackage{amssymb}
\usepackage{amsmath}
\usepackage{color}
\usepackage{bbm}
\usepackage{tikz}
\usepackage{enumerate}
\usepackage{bbm}

\newcommand{\R}{\mathbb R}
\newcommand{\N}{\mathbb N}

\newcommand{\E}{\mathbb E}
\newcommand{\Pro}{\mathbb P}
\newcommand{\dif}{\,\mathrm{d}}

\newtheorem{thm}{Theorem}[section]

\newtheorem{lemma}[thm]{Lemma}

\newtheorem{proposition}[thm]{Proposition}

\theoremstyle{remark}
\newtheorem{rmk}[thm]{Remark}

\allowdisplaybreaks
\begin{document}

%%%%%%%%%%%%%%%%%%%%%%%%%%%%%%%%%%%%%%%%%%%%%5

\title[Probabilistic Estimates for Tensor Products of Random Vectors]{Probabilistic Estimates for Tensor Products of Random Vectors}

\author[D. Alonso-Guti\'errez]{David Alonso-Guti\'errez}
\address{Departament de Matem\`atiques, Universitat Jaume I, Campus de Riu Sec, E12071 Castell\'o de la Plana, Spain}
\email{alonsod@uji.es}

\author[M. Passenbrunner]{Markus Passenbrunner}
\address{Institute of Analysis, Johannes Kepler University Linz,
Altenbergerstra{\ss}e 69, 4040 Linz, Austria} \email{markus.passenbrunner@jku.at}

\author[J. Prochno]{Joscha Prochno}
\address{Institute of Analysis, Johannes Kepler University Linz,
Altenbergerstra{\ss}e 69, 4040 Linz, Austria} \email{joscha.prochno@jku.at}


\thanks{D. Alonso is partially supported by Instituto de Matem\'aticas y Aplicaciones de Castell\'on, MINECO project MTM2013-42105-P, and BANCAJA project P1-1B2014-35. M. Passenbrunner is supported by the Austrian Science Fund, FWF P23987 and P27723. J. Prochno is supported by the Austrian Science Fund, FWFM 1628000.}

\date{\today}

\begin{abstract}
We prove some probabilistic estimates for tensor products of random vectors, generalizing results that were obtained by Gordon, Litvak, Sch\"utt, and Werner [Ann. Probab., 30(4):1833--1853, 2002], and Prochno and Riemer [Houst. J. Math., 39(4):1301--1311, 2013]. As an application we obtain embeddings of certain matrix spaces into $L_1$.
\end{abstract}

\maketitle

%\tableofcontents

%%%%%%%%%%%%%%%%%%%
%%%%%%%%%%%%%%%%%%%
\section{Introduction}
%%%%%%%%%%%%%%%%%%%
%%%%%%%%%%%%%%%%%%%

In \cite{KS1} and \cite{KS2} Kwapie\'n and Sch\"utt studied combinatorial and probabilistic inequalities related to Orlicz norms to investigate certain invariants of Banach spaces such as the positive projection constant of a finite-dimensional Orlicz space and to characterize the symmetric sublattices of $\ell_1(c_0)$ as well as the finite-dimensional symmetric subspaces of $\ell_1$.

Building upon that, in the last decade these techniques initiated further research, were extended, and successfully used in several different areas of mathematics. Those include the local theory of Banach spaces, when studying symmetric subspaces of $L_1$ \cite{RS,S2,S,PS,P, P1}, probability theory, to obtain uniform estimates for order statistics \cite{GLSW2} (see also \cite{GLSW1,GLSW,GLSW3,LPP}) as well as converse results on the distribution of random variables in connection with Musielak-Orlicz norms \cite{ACPP}, or convex geometry, to obtain sharp bounds for several geometric functionals on random polytopes \cite{AGP, AGP2,AGP3} such as the support function, the mean width and mean outer radii.

Let $X_1,\dots,X_n$ be independent copies of an integrable random variable $X$. In \cite[Lemma 5.2]{GLSW2}, it was proved that, if we define the Orlicz function $M_X$ by
\[
  M_{X}(s) = \int_0^s \int_{|X|\geq 1/t} |X| \dif\Pro \dif t,\qquad s\geq 0,
\]
then,
\begin{align}\label{eq: GLSW expectation of max}
\mathbb E \max\limits_{1\leq i \leq n}|a_i X_i|\simeq \|(a_i)_{i=1}^n\|_{M_{X}},\qquad a\in\R^n.
\end{align}

In the following, let $\xi_1,\dots,\xi_n$ be independent copies of an integrable random variable $\xi$ such that the collection $(\xi_1,\dots,\xi_n,X_1,\dots,X_n)$ is a family of independent random variables, and let $(a_{ij})_{i,j=1}^n\in\R^{n\times n}$.
In view of \eqref{eq: GLSW expectation of max} it is a natural question to ask whether we can find estimates for
\begin{equation}\label{eqn:ausdruck}
\mathbb E_{\xi}\mathbb E_{X}\max\limits_{1\leq i,j \leq n}\left|a_{ij}\xi_iX_j\right|,
\end{equation}
for these expressions naturally appear, for example, in the study of certain matrix subspaces of $L_1$. Note that, since the random variables $\xi_i\cdot X_j$, $i,j=1,\ldots,n$ are no longer independent on the product probability space, \cite[Lemma 5.2]{GLSW2} cannot be applied.

In \cite{PR}, among other things, sharp estimates (up to constants independent of the dimension $n$) in the case of $p$- and $q$-stable random variables ($q<p$) as well as for Gaussians were obtained. To be more precise, it was shown that if $\xi$ is a $q$-stable and $X$ a $p$-stable random variable, then
\begin{equation}\label{eq:rp 1}
\mathbb E_{X}\mathbb E_{\xi}\max\limits_{1\leq i,j \leq n}\left|a_{ij}\xi_i X_j\right|
\simeq_{p,q} \big\|\big(\|\left(a_{ij}\right)_{j=1}^n\|_p\big)_{i=1}^n\big\|_q.
\end{equation}
Note that instead of $r$-stable random variables one can choose $\log \gamma_{1,r}$ distributed random variables or random variables with density $r(r-1)x^{-r-1}\mathbbm 1_{[(r-1)^{1/r},\infty)}(x)$, since only the tail behavior is important. Those random variables give
\begin{equation}\label{eq:generation r norm max}
\E_\xi\max_{1\leq i \leq n}|a_i\xi_i| \simeq_r \|a\|_r.
\end{equation}
The advantage of the two latter distributions over an $r$-stable one is that we do not need to restrict ourselves to parameters $r\leq 2$.

In the second case, where $X$ is a standard Gaussian, the authors proved that
\begin{equation}\label{eq:rp 2}
\mathbb E_{X}\mathbb E_{\xi}\max\limits_{1\leq i,j \leq n}\left|a_{ij}\xi_iX_j\right|
\simeq_q \big\|\big(\|\left(a_{ij}\right)_{j=1}^n\|_{M_{X}}\big)_{i=1}^n\big\|_q,
\end{equation}
where $M_{X}$ is a suitable Orlicz function.

However, the case of arbitrarily distributed random variables is not covered in that work. This paper serves two purposes. On one hand, to fill that gap and provide estimates of the same flavor as \eqref{eq:rp 1} and \eqref{eq:rp 2}, but for arbitrary distributions of $X$. On the other hand, we consider the converse setting in which a certain Orlicz norm is given and we find a distribution of a random variable $X$ that corresponds to this norm. In addition, we study a more general setting, namely, expressions of the form
\[
\E_{\xi}\E_{X}\Big( \sum_{i,j}|a_{ij}\xi_iX_j|^p\Big)^{1/p},\qquad 1<p \leq \infty.
\]
In the special case that $p=2$, these expressions naturally appear in Banach space theory when studying the local structure of $L_1$. Although there are a number of sophisticated criteria at hand, to decide whether a given Banach space is a subspace of $L_1$ might still be non-trivial. In fact, it is well known that the finite-dimensional symmetric subspaces of $L_1$ are averages of 2-concave Orlicz spaces \cite{KS1} (see \cite{BDC} for the infinite-dimensional version), but, as can be seen in the case of Lorentz spaces, this is not easy to apply (cf., \cite{S2}). Nowadays it is still an open question what these symmetric subspaces of $L_1$ really are and a goal of Banach space theory to find characterizations that can be easily applied. While improving on the results from \cite{PR}, we also hope to provide a better understanding of the techniques as well as new estimates on the way to achieve that goal.

In the following, an Orlicz function $M$ is called normalized if
\[
\int_0^{\infty} x\dif M'(x)=1,
\]
where $M'$ is the right derivative of $M$.

The first main theorem of this work is the following:

\begin{thm}\label{thm:generating q of M norm}
Let $1<q<p\leq\infty$. Let $M\in \mathcal C^3$ be a normalized Orlicz function with $M'(0)=0$ that is linear on $[M^{-1}(1),\infty)$ and satisfies
\begin{equation}\label{eq:limit exists}
\lim_{t\to 0^+}\frac{M''(t)}{t^{q-2}}=0 \text{ exists (in $\overline{\R}$)},
\end{equation}
\begin{equation}\label{eq:integral 3}
\int_0^s \frac{M(t)}{t^q} \frac{dt}{t} \leq C \cdot \frac{M(s)}{s^q},\qquad  0< s\leq M^{-1}(1),
\end{equation}
\begin{equation}\label{eq:density with M}
f(s):=\Big(1-\frac{2}{p}\Big)s^{-3}M''(s^{-1})-\frac{1}{p}s^{-4}M'''(s^{-1}) \text{ is non-negative for all $s> 0$}.
\end{equation}
Then, $f$ is a probability density and for all $(a_{ij})_{i,j=1}^n\in\R^{n\times n}$,
\[
\E_{\xi} \E_X \| (a_{ij}\xi_iX_j)_{i,j=1}^n\|_p \simeq_{p,q}  \big\| \big( \| (a_{ij})_{j=1}^n \|_{M}\big)_{i=1}^n \big\|_q,
\]
where $(\xi_i)_{i=1}^n$, $(X_j)_{j=1}^n$ are independent collections of independent copies of random variables $\xi$ and $X$ with densities $f_{\xi}(x)=q(q-1)x^{-q-1}\mathbbm 1_{[(q-1)^{1/q},\infty)}(x)$ and $f_{X}=f$ respectively.

\end{thm}

We will clarify the meaning of \eqref{eq:integral 3} by presenting an equivalent pointwise inequality related to the well-known $\Delta_2$-condition in Section \ref{discussion integral condition}. Moreover, we will see that condition \eqref{eq:integral 3} implies that the limit in \eqref{eq:limit exists} is zero if it exists.

If independent copies $X_1,\ldots,X_n$ of a $q$-integrable random variable $X$ are given ($1<q<p\leq\infty$), where $|X|$ has a continuous density, then we obtain the following theorem in the flavor of the results in \cite{GLSW2, ACPP, PR} (cf., Theorems \ref{thm:orlicz_carsten1}, \ref{thm:orlicz_p_norm} below, and the discussion above).

\begin{thm}\label{cor:reformulation main theorem}
Let $1<q<p\leq\infty$ and $X_1,\ldots,X_n$ be independent copies of a $q$-integrable random variable $X$, where $|X|$ has a continuous density. For all $s> 0$ let
\[
M_{X,p}(s) = \frac{p}{p-1}\int_0^s\bigg( \int_{|X| \leq 1/t} t^{p-1} \left| X \right|^p \dif \Pro + \int_{ |X| \geq 1/t}|X| \dif \mathbb P \bigg)\dif t ,
\]
or, if $p=\infty$,
\[
M_X(s)=M_{X,\infty}(s)= \int_0^s \int_{|X|\geq 1/t}|X| \dif\Pro \dif t.
\]
Assume that
\begin{equation}\label{eq:limit exists M_X,p}
\lim_{t\to 0^+}\frac{M_{X,p}''(t)}{t^{q-2}} \text{ exists (in $\overline{\R}$)},
\end{equation}
\begin{equation}\label{eq:integral 2}
\int_0^s \frac{M_{X,p}(t)}{t^q} \frac{dt}{t} \leq C \frac{M_{X,p}(s)}{s^q}, \qquad  \,0 < s \leq M_{X,p}^{-1}(1),
\end{equation}
Then, for all $(a_{ij})_{i,j=1}^n\in\R^{n\times n}$,
\[
 \E_{\xi} \E_X \| (a_{ij}\xi_iX_j)_{i,j=1}^n\|_p \simeq_{p,q} \big\| \big( \| (a_{ij})_{j=1}^n \|_{M_{X,p}}\big)_{i=1}^n \big\|_q,
\]
where $\xi_1,\ldots,\xi_n$ are independent copies of a random variable $\xi$ with density $f_{\xi}(x)=q(q-1)x^{-q-1}\mathbbm 1_{[(q-1)^{1/q},\infty)}(x)$.
\end{thm}

The proofs of these theorems will be carried out for $p<\infty$. For $p=\infty$ they are very similar and will be omitted here. The statements of the theorems therefore remain true if we formally set $p=\infty$. 

The paper is organized as follows: in Section \ref{Preliminaries} we present some basic facts about Orlicz functions, known results and their consequences that will be used throughout the paper. In Section \ref{discussion integral condition} we will discuss the conditions imposed on the Orlicz function in Theorem \ref{thm:generating q of M norm}. In Section \ref{ell p} we will prove Theorems \ref{thm:generating q of M norm} and \ref{cor:reformulation main theorem} in the case $1<p<\infty$. In Section \ref{ApplicationBanach} we will show an application of our results to Banach space theory in which we show how to use Theorem \ref{thm:generating q of M norm} to give a uniform isomorphic embedding of certain matrix spaces into $L_1$. Finally, in the last section, we generalize a result that was obtained in \cite{ACPP} (Theorem \ref{thm:orlicz_p_norm} in this paper) to an arbitrary Orlicz norm instead of the $\ell_p$-norm  (Theorem \ref{generalization}). Note that this is a simplification of \cite[Theorem 1]{GLSW1}.
%%%%%%%%%%%%%%%%%%%
%%%%%%%%%%%%%%%%%%%
\section{Preliminaries}\label{Preliminaries}
%%%%%%%%%%%%%%%%%%%
%%%%%%%%%%%%%%%%%%%

A convex function $M:[0,\infty)\rightarrow[0,\infty)$ with $M(t)>0$ for $t>0$ and $M(0)=0$ is called an Orlicz function. For an Orlicz function $M$ we define the Luxemburg norm $\|\cdot\|_M$ on $\R^n$ by
$$
  \|x\|_M=\inf\left\{t>0 \,:\, \sum\limits_{i=1}\limits^nM\left(\frac{|x_i|}{t}\right)\leq 1\right\},
$$
and the Orlicz space $\ell_M^n$ to be the vector space $\R^n$ equipped with the norm $\|\cdot\|_M$. Moreover, a Luxemburg norm $\|\cdot\|_M$ is uniquely determined by the values of $M$ on the interval $[0,M^{-1}(1)]$. We say that an Orlicz function $M$ is normalized if
\[
\int_0^\infty x\dif M'(x)=1,
\]
where $M'$ is the right derivative of $M$.
We say that two Orlicz functions $M$ and $N$ are equivalent
if there are positive constants $a$ and $b$ such that for all
$t\geq0$
\[
 a^{-1}M( b^{-1}t) \leq N(t) \leq aM(bt).
\]
If two Orlicz functions are equivalent, so are their norms.
For a detailed and thorough introduction to Orlicz spaces see, e.g., \cite{RR} or \cite{LT77}.

Let $X$ and $Y$ be isomorphic Banach spaces. We say that they are
$C$-isomorphic if there is an isomorphism $T:X\rightarrow Y$ with
$\|T\|\|T^{-1}\|\leq C$.
We define the Banach-Mazur distance of $X$ and $Y$ by
    \[
      d(X,Y) = \inf\left\{ \|T\|\|T^{-1}\| \,:\, T\in L(X,Y) ~ \hbox{isomorphism} \right\}.
    \]
 Let $(X_n)_n$ be a sequence of $n$-dimensional normed spaces and let $Z$ be a normed space. If there exists a constant $C>0$ such that for all $n\in\N$ there exists a subspace $Y_n \subseteq Z$ with $\dim(Y_n)=n$ and $d(X_n,Y_n)\leq C$, then we say $(X_n)_n$ embeds uniformly into $Z$. The monograph \cite{TJ} gives a detailed introduction to the concept of Banach-Mazur distances.

Throughout this paper, we will write $A(t) \simeq B(t)$ to denote that there are absolute constants $c_1$ and $c_2$ such that $c_1 A(t) \leq B(t) \leq c_2 A(t)$ for all $t$, where $t$ denotes all implicit and explicit dependencies that the expressions $A$ and $B$ might have.
If the constants depend on a certain parameter $p$, we denote this by $\simeq_p$. By $c,C...$ we denote positive absolute constants. We write $c_p,C_p$ if the constants depend on some parameter $p$. The value of the constants may change from line to line.

In \cite[Lemma 5.2]{GLSW2}, the authors proved the following theorem:

\begin{thm}\label{thm:orlicz_carsten1} Let $X_1,\dots,X_n$ be independent copies of an integrable random variable $X$. For all $s\geq 0$ define
\begin{equation}\label{eq:orlicz_carsten1}
M_X(s)=\int_0^s\int_{|X|\geq 1/t} |X| \, \dif \Pro \dif t.
\end{equation}
Then, for all vectors $a=(a_i)_{i=1}^n\in\R^n$,
\[
c_1\|a\|_{M_{X}} \leq \mathbb E \max\limits_{1\leq i \leq n}|a_i X_i| \leq c_2 \|a\|_{M_{X}},
\]
where $c_1,c_2$ are absolute constants.
\end{thm}

Note that $M_X$ as defined in \eqref{eq:orlicz_carsten1} is non-negative, convex, and can be written in the following way:
\begin{equation}\label{eqn:m1}
	  M_{X}(s) = s\int_{1/s}^\infty x\dif \Pro_{|X|}(x) - \Pro(|X|\geq 1/s).
\end{equation}
Moreover, in many cases, $M_X$ is normalized and $M_X(0)=M'_X(0)=0$. For instance, this is the case if $\Pro_{|X|}$ is absolutely continuous with respect to Lebesgue measure.

The next proposition (cf., \cite[Proposition 4.1]{ACPP}) is a converse result to Theorem~\ref{thm:orlicz_carsten1}:

\begin{proposition}\label{PRO_inverse_maximum}
Let $M$ be a normalized Orlicz function with $M'(0)=0$ such that $\int_0^\infty\dif M^\prime(s)$ is finite. Let $X_1,\ldots,X_n$ be independent copies of the random variable $X$ with distribution
\begin{equation}\label{eq:distrX}
\mathbb P(X\geq	 t)=\int_{[0,1/t]} s\dif M'(s),\qquad t> 0.
\end{equation}
Then, for all $x=(x_i)_{i=1}^n\in \mathbb{R}^n$,
\[
c_1\|x\|_M \leq \mathbb{E}\max_{1\leq i\leq n}|x_iX_i| \leq c_2 \|x\|_M,
\]
where $c_1,c_2$ are constants independent of the Orlicz function $M$.
\end{proposition}

\begin{rmk}
Note that in the previous proposition $X$ is integrable if and only if $\int_0^\infty\dif M'(s)$ is finite. This is the case, for instance, if $M'$ is absolutely continuous and $M''$ is integrable.
\end{rmk}

If $M$ is ``sufficiently smooth'', we get that the density $f_X$ of $X$ is given by
\begin{equation}\label{eq:density_smooth_M}
f_X(t)={t^{-3}}M''(t^{-1}).
\end{equation}
To generate an $\ell_p$-norm in Proposition \ref{PRO_inverse_maximum}, i.e., to consider the case $M(t)=t^p$, one needs to pass to an equivalent Orlicz function so that the normalization condition is satisfied. The function $\widetilde M$ with $\widetilde M(t) = t^p$ on $[0, (p-1)^{-1/p}]$ which is then extended linearly does the trick.

Note that the assumption for $M$ to be normalized in Proposition \ref{PRO_inverse_maximum} is natural, since in many cases the function $M_X$ of Theorem \ref{thm:orlicz_carsten1} is.
%This holds since for all $s>0$
%\[
%M_X''(s) = \frac{1}{s^3} f_X(s^{-1}),
%\]
%where $f_X$ is the density function of $X$, and thus
%\[
%\int_0^\infty x^{-2} f_X(x^{-1}) \dif x = \int_0^\infty f_X(x) \dif x =1.
%\]

The next result was recently obtained in \cite[Theorem 1.1]{ACPP} and holds in the more general setting of Musielak-Orlicz spaces, though we only state it here for Orlicz spaces:

\begin{thm} \label{main}
Let $1<p<\infty$ and $M\in\mathcal{C}^3$ be a normalized Orlicz function that is linear on $[M^{-1}(1),\infty)$ and satisfies $M'(0)=0$.
Moreover, assume that for all $x>0$
\[
f(x)=\Big(1-\frac{2}{p}\Big)x^{-3}M''(x^{-1})-\frac{1}{p}x^{-4}M'''(x^{-1}) \text{ is non-negative}.
\]
Then $f$ is a probability density and for all $x\in\R^n$,
\[
c_1(p-1)^{1/p}\|x\|_{M} \leq \E \|(x_i X_i)_{i=1}^n\|_p \leq c_2\|x\|_{M},
\]
where $c_1,c_2$ are positive absolute constants and $X_1,\dots,X_n$ are independent copies of a random variable $X$ with density $f_X=f$.
\end{thm}

Observe that in the latter theorem, for all $x>0$,
\begin{equation}\label{eq:distribution X}
\Pro (X\geq x) = - M(x^{-1}) + x^{-1} M'(x^{-1}) - \frac{1}{p}x^{-2} M''(x^{-1}).
\end{equation}

Another theorem that was obtained in \cite[Theorem 3.1]{ACPP} is the following:

\begin{thm} \label{thm:orlicz_p_norm}
Let $1<p<\infty$, $X_1,\ldots,X_n$ be independent copies of an integrable random variable $X$. For all $s\geq0$ define
\begin{equation}\label{M_X,p}
M_{X,p}(s) = \frac{p}{p-1}\int_0^s\bigg[ \int_{|X| \leq 1/t} t^{p-1} \left| X \right|^p \dif \Pro + \int_{ |X| > 1/t}|X| \dif \mathbb P \bigg]\dif t .
\end{equation}
Then, for all $x=(x_i)_{i=1}^n\in\R^n$,
\[
 c_1 (p-1)^{1/p} \| x \|_{M_{X,p}} \leq \mathbb E \| (x_iX_i)_{i=1}^n \|_p \leq c_2 \| x \|_{M_{X,p}},
\]
where $c_1,c_2,$ are positive absolute constants.
\end{thm}

Obviously, $M_{X,p}(0)=0$ and, as can be checked by direct computation, the integrand in \eqref{M_X,p}, i.e., the function
\[
t\mapsto \int_{|X| \leq 1/t} t^{p-1} \left| X \right|^p \dif \Pro + \int_{ |X| > 1/t}|X| \dif \mathbb P,
\] 
is increasing in $t$. Therefore, the function $M_{X,p}$ is indeed an Orlicz function.

%in the proof of the theorem (cf., \cite[Theorem 3.1]{ACPP}), $M_{X,p}(s) = M_{X\xi}$ with $M_{X\xi}$ given as in \eqref{eq:orlicz_carsten1} and where $\xi$ is distributed according to $f_\xi(x) = px^{-p-1}\mathbbm 1_{[1,\infty)}(x)$. Thus, $M_{X,p}$ is indeed an Orlicz function.

Note that, using Fubini's theorem, we can rewrite the function $M_{X,p}$ as
\begin{equation}\label{eq:rewritten M by fubini}
M_{X,p}(s) =  \frac{s^p}{p-1} \int_0^{1/s} x^p \dif \Pro_{|X|}(x) + \frac{p}{p-1}\cdot s \int_{1/s}^\infty x \dif \Pro_{|X|}(x) - \Pro(|X| \geq s^{-1}).
\end{equation}

In the last section, we will generalize Theorem \ref{thm:orlicz_p_norm} to a setting where the $\ell_p$-norm is replaced by an arbitrary Orlicz norm.

A natural question concerning Theorems \ref{main} and \ref{thm:orlicz_p_norm} is: given an Orlicz function $M$ according to Theorem \ref{main} and the generating distribution \eqref{eq:distribution X} of a random variable $X$, is the function $M_{X,p}$ in \eqref{M_X,p} equivalent to $M$? As it turns out, the functions are not only equivalent, but we even have $M_{X,p}=M$. We state this in the following lemma:

\begin{lemma}\label{lem:M_X,p equal M}
Let $1<p\leq\infty$ and $M$ be as in Theorem \ref{main}.
Assume that $X$ is a random variable with density
\begin{equation}\label{eq:density}
f_X(x)=\Big(1-\frac{2}{p}\Big)\frac{1}{x^3}M''\Big(\frac{1}{x}\Big)-\frac{1}{px^4}M'''\Big(\frac{1}{x}\Big),\qquad x>0.
\end{equation}
Then, for all $s\geq 0$,
\[
M_{X,p}(s) = M(s),
\]
where $M_{X,p}$ is given by \eqref{M_X,p} for $1<p<\infty$ and by \eqref{eq:orlicz_carsten1} for $p=\infty$.
\end{lemma}
The proof of this result is a straightforward calculation using integration by parts, the definition of $M_{X,p}$, and \eqref{eq:density}.
%\begin{proof}
%Integration by parts shows that for any $0< t \leq s$
%\[
%\int_{0}^{t^{-1}}t^{p-1}x^p \dif \Pro_X(x) = \frac{1}{p}tM''(t)
%\]
%and
%\[
%\int_{t^{-1}}^\infty x \dif \Pro_X(x) = (1-p^{-1})M'(t) - \frac{1}{p}tM''(t).
%\]
%Thus, for all $s\geq 0$,
%\[
%M_{X,p}(s) = \frac{p}{p-1}\int_{0}^s (1-p^{-1})M'(t) \dif t = M(s).\qedhere
%\]
%\end{proof}

The following result, which relates the density of a given random variable with its associated Orlicz function is, in a certain way, a converse to the previous lemma:

\begin{lemma}\label{lem:density of given X}
Let $1<p\leq\infty$ and $X$ be an integrable random variable such that $|X|$ has continuous density $f_{|X|}$. Then $M_{X,p}\in\mathcal C^3$ and
\[
f_{|X|}(x)= \Big(1-\frac{2}{p}\Big)\frac{1}{x^3}M_{X,p}''\Big(\frac{1}{x}\Big)-\frac{1}{px^4}M_{X,p}'''\Big(\frac{1}{x}\Big),\qquad x>0,
\]
where $M_{X,p}$ is given by \eqref{M_X,p} for $1<p<\infty$ and by \eqref{eq:orlicz_carsten1} for $p=\infty$.
\end{lemma}
\begin{proof}
First note that since $f_{|X|}$ is continuous, $M_{X,p}$ is at least twice continuously differentiable.
Assume $1<p<\infty$. Then, for all $s\geq 0$,
\[
M_{X,p}(s) =  \frac{p}{p-1}\int_0^s\bigg[ \int_{0}^{1/t} t^{p-1} r^p f_{|X|}(r) \dif r + \int_{1/t}^{\infty} r f_{|X|}(r) \dif r \bigg]\dif t,
\]
\[
M_{X,p}'(s) =  \frac{p}{p-1}\bigg[ s^{p-1} \int_{0}^{1/s}  r^p f_{|X|}(r) \dif r + \int_{1/s}^{\infty} r f_{|X|}(r) \dif r \bigg],
\]
\[
M_{X,p}''(s) =  ps^{p-2} \int_{0}^{1/s}  r^p f_{|X|}(r) \dif r.
\]
Hence, $M_{X,p}\in\mathcal C^3$ and
\[
M_{X,p}'''(s) =  p(p-2)s^{p-3} \int_{0}^{1/s}  r^p f_{|X|}(r) \dif r - ps^{-4}f_{|X|}(1/s).
\]
Therefore, combining the equalities, we find that for all $x>0$
\[
f_{|X|}(x)= \Big(1-\frac{2}{p}\Big)\frac{1}{x^3}M_{X,p}''\Big(\frac{1}{x}\Big)-\frac{1}{px^4}M_{X,p}'''\Big(\frac{1}{x}\Big).
\]
Similarly, we obtain the result for $p=\infty$.
\end{proof}

%%%%%%%%%%%%%%%%%%%%%%%%%%%%%%%%%%%%%%%%%%%%%
%%%%%%%%%%%%%%%%%%%%%%%%%%%%%%%%%%%%%%%%%%%%%
\section{Discussion of the conditions}\label{discussion integral condition}
%%%%%%%%%%%%%%%%%%%%%%%%%%%%%%%%%%%%%%%%%%%%%
%%%%%%%%%%%%%%%%%%%%%%%%%%%%%%%%%%%%%%%%%%%%%

The integral condition \eqref{eq:integral 3} that appears in Theorem \ref{thm:generating q of M norm} can be interpreted as a growth condition on $M$, related to the well known $\Delta_2$ condition. To be more precise, we have the following:

\begin{proposition}\label{prop:pointwise consition}
Let $1\leq q < \infty$ and $M$ be an Orlicz function. Then \eqref{eq:integral 3} holds if and only if there exists a constant $c<1$ and some $\gamma<1$ such that for all $s\geq 0$
\begin{equation}\label{eq:pointwise condition}
M(cs) \leq \gamma c^qM(s).
\end{equation}
\end{proposition}
\begin{proof}
We first show that \eqref{eq:pointwise condition} implies \eqref{eq:integral 3}. If $s\geq 0$, then
\[
\int_0^s\frac{M(t)}{t^{q+1}} \dif t = \sum_{k=0}^\infty \int_{c^{k+1}s}^{c^ks} \frac{M(t)}{t^{q+1}} \dif t=\sum_{k=0}^\infty \int_{cs}^{s} \frac{M(c^ku)}{(c^ku)^{q+1}}c^k \dif u.
\]
Using \eqref{eq:pointwise condition} inductively, we deduce the inequality
\[
\int_0^s\frac{M(t)}{t^{q+1}} \dif t \leq \sum_{k=0}^\infty \int_{cs}^{s} \frac{\gamma^kc^{kq}M(u)}{(c^ku)^{q+1}}c^k \dif u = \left(\sum_{k=0}^{\infty}\gamma^k\right) \int_{cs}^{s} \frac{M(u)}{u^{q+1}} \dif u.
\]
Since this last integral can be estimated by
\[
\int_{cs}^{s} \frac{M(u)}{u^{q+1}} \dif u \leq \frac{sM(s)}{(cs)^{q+1}}=c^{-q-1}\frac{M(s)}{s^q},
\]
the implication $\eqref{eq:pointwise condition}\Rightarrow \eqref{eq:integral 3}$ is proved (with constant $C=c^{-q-1}$).

Now we prove that \eqref{eq:integral 3} implies \eqref{eq:pointwise condition}.
Let $s\geq 0$. Then we obtain from \eqref{eq:integral 3}
\[
C\frac{M(s)}{s^{q}} \geq \sum_{k=0}^\infty \int_{2^{-k-1}s}^{2^{-k}s}\frac{M(t)}{t^{q+1}} \dif t \geq \sum_{k=0}^\infty2^{-k-1} \frac{sM(2^{-k-1}s)}{2^{-k(q+1)}s^{q+1}} = \frac{1}{2s^q} \sum_{k=0}^{\infty} 2^{kq}M(2^{-k-1}s).
\]
Thus
\begin{equation}\label{equationincondition}
\sum_{k=0}^\infty 2^{kq}M(2^{-k-1}s) \leq 2C M(s).
\end{equation}
Using this inequality twice, we see
\[
M(s) \geq \frac{1}{2C} \sum_{k=0}^\infty 2^{kq}M(2^{-k-1}s) \geq \frac{1}{(2C)^2} \sum_{k=0}^\infty 2^{kq} \sum_{\ell=0}^\infty2^{\ell q} M(2^{-k-\ell-2 }s).
\]
Observe that by rearranging the sums,
\[
\sum_{k=0}^\infty 2^{kq} \sum_{\ell=0}^\infty2^{\ell q} M(2^{-k-\ell-2 }s) = \sum_{r=0}^\infty \sum_{(k,\ell): k+\ell=r} 2^{qr}M(2^{-r-2}s).
\]
Hence
\begin{align*}
M(s) & \geq \frac{1}{(2C)^2}\sum_{r=0}^\infty \sum_{(k,\ell):k+\ell=r} 2^{qr}M(2^{-r-2}s) = \frac{1}{(2C)^2}\sum_{r=0}^\infty (r+1) 2^{qr}M(2^{-r-2}s) .
\end{align*}
Thus, choosing $r_0>0$ such that $\gamma^{-1}:=\frac{r_0+1}{(2C)^2}2^{-2q}>1$ ($\gamma$ as in in \eqref{eq:pointwise condition}) and taking $c=2^{-r_0-2}$ (which depends on $C$ and $q$), we obtain condition \eqref{eq:pointwise condition} by estimating the latter sum from below by the term with index $r_0$.
\end{proof}

Recall that an Orlicz function $M$ satisfies the $\Delta_2$-condition if and only if
\[
M(Ks) \leq C_K M(s),
\]
where $K$ can be any number larger than $1$ and $C_K$ is a constant only depending on $K$.
The $\Delta_2$-condition at zero is equivalent to $\ell_M$ being separable on the one hand, and to the fact that the standard unit vectors form a boundedly complete symmetric basis of $\ell_M$ on the other.

Since \eqref{eq:pointwise condition} is equivalent to
\[
M^*(\gamma^{-1}c^{-q+1}s) \leq \gamma^{-1}c^{-q}M^*(s),\qquad s\geq 0,
\]
where $M^*$ is the conjugate function of $M$ (cf., \cite[Theorem 4.2]{KR}), we can interpret the integral condition \eqref{eq:integral 3} as a special $\Delta_2$-condition on $M^*$, where $C_K$ satisfies some kind of homogeneity condition of degree $q^*=q/(q-1)$. Note that, in order to carry out this duality argument formally, $M$ needs to be an $N$-function, where an $N$-function is an Orlicz function $M$ such that additionally
\[
\lim_{s\to 0}\frac{M(s)}{s}=0\quad\hbox{and}\quad \lim_{s\to \infty}\frac{M(s)}{s}=\infty.
\]

%\begin{rmk}\label{rmk:convexity and integral condition}
If $M$ is an Orlicz function such that $t\mapsto M(t)t^{-q-\varepsilon}$, $\varepsilon>0$, $q>1$ is an increasing function, then condition \eqref{eq:integral 3} is satisfied. Note that in particular $(q+\varepsilon)$-convexity of $M$, i.e., the convexity of $t\mapsto M(t^{(q+\varepsilon)^{-1}})$, implies inequality \eqref{eq:integral 3} (with constant $C(\varepsilon)=\varepsilon^{-1}$). However, it seems that neither of them implies the other one.
%\end{rmk}

Let us now briefly discuss the limit conditions in Theorems \ref{thm:generating q of M norm} and \ref{cor:reformulation main theorem}. We will frequently make use of the fact that conditions \eqref{eq:limit exists} and \eqref{eq:integral 3} imply that the limit in \eqref{eq:limit exists} and corresponding limits for the derivatives must be $0$. We state this in the following lemma:

\begin{lemma}\label{lem:limits exists and are zero}
Let $1<q<\infty$ and $M\in\mathcal C^2$ be an Orlicz function satisfying $M(0)=0=M'(0)$ such that $\lim_{t\to 0^+}M''(t)/t^{q-2}$ exists and
\[
\int_0^s \frac{M(t)}{t^{q+1}} \dif t \leq C \frac{M(s)}{s^q},\qquad 0<s\leq M^{-1}(1).
\]
Then
\[
\lim_{t\to 0^+} \frac{M(t)}{t^{q}} = \lim_{t\to 0^+} \frac{M'(t)}{t^{q-1}} = \lim_{t\to 0^+} \frac{M''(t)}{t^{q-2}} = 0.
\]
\end{lemma}
\begin{proof}
By L'Hospital's rule and the existence of $\lim_{t\to 0^+}M''(t)/t^{q-2}$, the following limits exist and are equal:
\[
\lim_{t\to 0^+} \frac{M(t)}{t^{q}}= \lim_{t\to 0^+} \frac{M'(t)}{qt^{q-1}}= \lim_{t\to 0^+} \frac{M''(t)}{q(q-1)t^{q-2}}.
\]
Notice that for any $0<s\leq M^{-1}(1)$
\[
\frac{1}{s}\int_0^s \frac{M(t)}{t^q}\dif t\leq \int_0^s \frac{M(t)}{t^q}\frac{\dif t}{t}\leq C\frac{M(s)}{s^q}.
\]
Since the last expression is finite, we have
\[
\frac{1}{s}\int_0^s \frac{M(t)}{t^q}\dif t\leq \int_0^s \frac{M(t)}{t^q}\frac{\dif t}{t} < \infty.
\]
The mean value theorem for integrals implies that there exists a point $\xi=\xi(s)\in (0,s]$ such that
\[
\frac{M(\xi)}{\xi^q}=\frac{1}{s}\int_0^s \frac{M(t)}{t^q}\dif t.
\]
Taking limits (since this limit exists), we get
\[
\lim_{s\to 0^+} \frac{M(s)}{s^q} = \lim_{s\to 0^+} \frac{M(\xi(s))}{\xi(s)^q} \leq \lim_{s\to 0^+}\int_0^s \frac{M(t)}{t^q}\frac{\dif t}{t}=0.
\]
Therefore,
\[
\lim_{t\to 0^+}\frac{M(t)}{t^q}=\lim_{t\to 0^+}\frac{M'(t)}{t^{q-1}}=\lim_{t\to 0^+}\frac{M''(t)}{t^{q-2}}=0.\qedhere
\]

\end{proof}

\section{The expected value of $\ell_p$-norms of tensor products of random vectors}\label{ell p}
%%%%%%%%%%%%%%%%%%%%%%%%%%%%%%%%%%%%%%%%%%%
%%%%%%%%%%%%%%%%%%%%%%%%%%%%%%%%%%%%%%%%%%%

Before we present the proofs of Theorems \ref{thm:generating q of M norm} and \ref{cor:reformulation main theorem} for $p<\infty$, we need a generalization of \eqref{eq:generation r norm max} from $p=\infty$ to arbitrary $p\in(1,\infty]$.

\begin{lemma}
Let $1<q<p<\infty$ and $f:(0,\infty)\to\R$ be defined as
\[
f(x) = q(q-1)x^{-q-1}\mathbbm 1_{[(q-1)^{1/q},\infty)}(x).
\]
Then $f$ is a probability density and for all $a\in\R^n$,
\begin{equation}\label{generating_q_norm}
\mathbb E \Big( \sum_{j=1}^n |a_j \xi_j|^p\Big)^{1/p} \simeq_{p,q} \| a\|_q,
\end{equation}
where $\xi_1,\dots,\xi_n$ are independent copies of a random variable $\xi$ with density $f_\xi=f$.
\end{lemma}
\begin{proof}
By Theorem \ref{thm:orlicz_p_norm}, for all $a\in\R^n$,
\begin{equation}\label{eq: generating M norm in p norm}
c_1(p-1)^{1/p}\|a\|_{M_{\xi,p}} \leq \E \|(a_j \xi_j)_{j=1}^n \|_p \leq c_2\|a\|_{M_{\xi,p}},
\end{equation}
where $c_1,c_2$ are positive absolute constants and the function $M_{\xi,p}$ is given by
\[
M_{\xi,p}(s)=\begin{cases}\left(1+\frac{q}{p-1}+\frac{q(q-1)}{(p-1)(p-q)}\right)s^q - \frac{q(q-1)^{p/q}}{(p-1)(p-q)} s^p,& \textrm{if } s\leq (q-1)^{-1/q},\cr\frac{pq}{(p-1)(q-1)^{1-\frac{1}{q}}}s - 1,& \textrm{if } s\geq (q-1)^{-1/q}.\cr\end{cases}
\]
This follows from \eqref{M_X,p} by direct computation if we insert $f_\xi$. It can be checked that in $[0,M_{\xi,p}^{-1}(1)]$, $M_{\xi,p}$ is equivalent to $x\mapsto x^q$ up to constants depending only on $p$ and $q$.
\end{proof}

Before we continue to prove a uniform estimate for the expected value of general $\ell_p$-norms of tensor products of random vectors, we state a simple lemma which follows by integration by parts:

\begin{lemma}\label{lem:distribution and integral}
Let $1< p < \infty$ and $M\in\mathcal C^3$ be a function such that
\[
f(x) = \left( 1-\frac 2 p\right)\frac{1}{x^3} M''(x^{-1}) - \frac{1}{px^4}M'''(x^{-1}),\qquad x>0,
\]
is the probability density of a random variable $X$.
Then, for all $0<a<b<\infty$ and all $r\in\R$, we have
\begin{align*}
\int_a^b x^r \dif \Pro_X(x) & = \frac{1}{p} \big( M''(b^{-1})b^{r-2} - M''(a^{-1})a^{r-2}\big)\cr
 &+ \left(1-\frac{r}{p}\right)\left(M'(a^{-1})a^{r-1} - M'(b^{-1})b^{r-1}\right) \cr
& + (1-r)\left(1-\frac{r}{p}\right) \big( M(b^{-1})b^r - M(a^{-1})a^r \big) \cr
&- (1-r)r\left(1-\frac{r}{p}\right) \int_{b^{-1}}^{a^{-1}} \frac{M(y)}{y^{r+1}}\dif y.
\end{align*}
\end{lemma}

\begin{lemma}\label{lem:X^q integrable}
Let $1<q<p<\infty$, $M$ be as in Theorem \ref{thm:generating q of M norm}, and $X$ be a positive random variable with density
\[
f(x) = \left( 1-\frac 2 p\right)\frac{1}{x^3} M''(x^{-1}) - \frac{1}{px^4}M'''(x^{-1}),\qquad x>0.
\]
Then $X^q$ is integrable.
\end{lemma}
\begin{proof}
Using Lemma \ref{lem:distribution and integral}, Lemma \ref{lem:limits exists and are zero}, condition \eqref{eq:integral 3}, and the fact that $M$ is linear in the interval $[M^{-1}(1),\infty)$ gives the result.
\end{proof}

\begin{proof}[Proof of Theorem \ref{thm:generating q of M norm} for $p<\infty$]
The lower bound follows using (\ref{generating_q_norm}), the triangle inequality, and Theorem \ref{main}. To prove the upper bound, we first observe that
\begin{align*}
\E_\xi \E_X \big\| \left( a_{ij}\xi_iX_j\right)_{i,j=1}^n \big\|_p & = \E_X \E_\xi \Big( \sum_{i=1}^n |\xi_i|^p \sum_{j=1}^n |a_{ij} X_j|^p \Big)^{1/p} \cr
& \simeq_{p,q} \E_X  \bigg\| \bigg( \Big(\sum_{j=1}^n |a_{ij} X_j|^p\Big)^{1/p}\bigg)_{i=1}^n \bigg\|_q,
\end{align*}
where we used \eqref{generating_q_norm}.
Thus, applying Jensen's inequality,
\begin{align*}
\E_X  \bigg\| \bigg( \Big(\sum_{j=1}^n |a_{ij} X_j|^p\Big)^{1/p}\bigg)_{i=1}^n \bigg\|_q
& \leq  \bigg( \sum_{i=1}^n \E_X \Big(\sum_{j=1}^n \big[|a_{ij}X_j|^q\big]^{p/q} \Big)^{q/p} \bigg)^{1/q}.
\end{align*}
By definition, $X$ is positive and by Lemma \ref{lem:X^q integrable}, $X^q$ is integrable. Hence, Theorem \ref{thm:orlicz_p_norm} applied to the random variables $X_1^q,\ldots, X_n^q$ and parameter $r=p/q$ (instead of parameter $p$) gives
\[
\E_X \big\| \big( |a_{ij}X_j|^q \big)_{j=1}^n\big\|_{ p/ q} \lesssim \big\| \big(|a_{ij}|^q\big)_{j=1}^n \big\|_{M_{X^q, p/ q}},
\]
where
\[
M_{X^q,p/q}(s) = \frac{p}{p-q}\int_0^s\bigg[ \int_{|X|^q \leq 1/t} t^{p/q -1} \left| X \right|^{p} \dif \Pro + \int_{ |X|^q > 1/t}|X|^q \dif \mathbb P \bigg]\dif t
\]
or, as we saw in (\ref{eq:rewritten M by fubini}),
\[
M_{X^q, p/q}(s) = \frac{q\cdot s^{p/q}}{p-q}  \int_0^{s^{-1/q}} x^p \dif \Pro_X(x) + \frac{p\cdot s}{p-q}\int_{s^{-1/q}}^\infty x^q \dif \Pro_X(x) - \Pro(X \geq s^{-1/q}).
\]
Since
\[
\big\| \big(|a_{ij}|^q\big)_{j=1}^n\big\|_{M_{X^q,p/q}} = \big\| \big(|a_{ij}|\big)_{j=1}^n\big\|_{M_{X^q,p/q}\circ t^q}^q,
\]
we obtain
\[
\E_\xi \E_X \big\| \big( a_{ij}\xi_iX_j\big)_{i,j=1}^n \big\|_p \lesssim_{p,q} \big\| \big( \| (a_{ij})_{j=1}^n\|_{M_{X^q,p/q}\circ t^q} \big)_{i=1}^n \big\|_q.
\]
Thus it is left to show that $M_{X^q,p/q}( s^q) \leq c_{p,q} M(s)$ for all $s\geq 0$. We have
\[
M_{X^q,p/q}(s^q) = \frac{q}{p-q} s^p \int_0^{1/s} x^p \dif \Pro_X(x) + \frac{p}{p-q} s^q \int_{1/s}^\infty x^q \dif \Pro_X(x) - \Pro(X \geq s^{-1}).
\]
By Lemma \ref{lem:distribution and integral}, the fact that $M$ is $\mathcal C^3$ and linear in the interval $[M^{-1}(1),\infty)$, by Lemma \ref{lem:limits exists and are zero}, and condition \eqref{eq:integral 3}, we obtain that
\begin{align*}
\frac{q}{p-q}s^p\int_0^{1/s} x^p \dif \Pro_X(x) & = \frac{q}{p(p-q)}s^2M''(s),\cr
\frac{p}{p-q} s^q \int_{1/s}^\infty x^q \dif \Pro_X(x) & = (q-1)M(s)+sM'(s)-\frac{1}{p-q}s^2M''(s) \cr
&\quad+q(q-1)s^q\int_0^s\frac{M(y)}{y^{q+1}\dif y},
\end{align*}
and
\[
\Pro(X \geq s^{-1}) = -M(s)+sM'(s)-\frac{1}{p}s^2M''(s).
\]
Combining these equalities, we find that
\[
M_{X^q,p/q}(s^q) = qM(s)+ q(q-1)s^q\int_0^s\frac{M(y)}{y^{q+1}\dif y}.
\]
By condition \eqref{eq:integral 3}, we can estimate this from above as
\[
M_{X^q,p/q}(s^q) \leq q(1+C(q-1))\cdot M(s).\qedhere
\]

\end{proof}

Let us now continue with the proof of Theorem \ref{cor:reformulation main theorem}, which follows the same lines as the previous proof.

\begin{proof}[Proof of Theorem \ref{cor:reformulation main theorem} for $p<\infty$]

From \eqref{generating_q_norm} we obtain
\[
\E_X\E_\xi \|(a_{ij}\xi_iX_j)_{i,j=1}^n\|_p \simeq_{p,q} \E_X\Bigg\| \Bigg(\Big(\sum_{j=1}^n |a_{ij}X_j|^p\Big)^{1/p} \Bigg)_{i=1}^n \Bigg\|_q.
\]
The lower bound follows from the triangle inequality and Theorem \ref{thm:orlicz_p_norm}.

For the upper bound we use Jensen's inequality and obtain
\[
\E_X\E_\xi \|(a_{ij}\xi_iX_j)_{i,j=1}^n\|_p \lesssim_{p,q} \left(\sum_{i=1}^n \E_X\Big\|\big(|a_{ij}X_j|^q\big)_{j=1}^n \Big\|_{p/q} \right)^{1/q}.
\]
Since, by assumption, $|X|^q$ is integrable, we have by Theorem \ref{thm:orlicz_p_norm} (used with parameter $p/q$) that
\begin{align*}
\E_X\E_\xi \|(a_{ij}\xi_iX_j)_{i,j=1}^n\|_p & \lesssim_{p,q} \left(\sum_{i=1}^n \Big\|\big(|a_{ij}|^q\big)_{j=1}^n \Big\|_{M_{|X|^q,p/q}} \right)^{1/q} \\
& = \left(\sum_{i=1}^n \Big\|\big(a_{ij}\big)_{j=1}^n \Big\|^q_{M_{|X|^q,p/q}\circ t^q} \right)^{1/q}.
\end{align*}
Again, $M_{|X|^q,p/q}\circ t^q$ is the Orlicz function $x\mapsto M_{|X|^q,p/q}(x^q)$. It is left to show that for all $s\geq 0$, $M_{|X|^q,p/q}(s^q) \leq C_{p,q}M_{X,p}(s)$. In order to prove it, we use that (by formula \ref{eq:rewritten M by fubini})
\[
M_{|X|^q,p/q}(s^q) = \frac{q}{p-q} s^p \int_0^{1/s} x^p \dif \Pro_{|X|}(x) + \frac{p}{p-q} s^q \int_{1/s}^\infty x^q \dif \Pro_{|X|}(x) - \Pro(|X| \geq s^{-1}).
\]
By Lemma \ref{lem:density of given X}, we can use Lemma \ref{lem:distribution and integral} to compute the integrals. Taking into account that by the expression obtained in the proof of Lemma \ref{lem:density of given X} for $M_{X,p}''$, we obtain that $\lim_{s\to\infty}M_{X,p}''(s)/ s^{p-2}=0$. Using this in combination with conditions \eqref{eq:limit exists M_X,p} and \eqref{eq:integral 2}, we obtain, as in the proof of Theorem \ref{thm:generating q of M norm}, that
\[
M_{|X|^q,p/q}(s^q) \leq q(1+C(q-1))\cdot M_{X,p}(s),
\]
which finishes the proof.
\end{proof}

%%%%%%%%%%%%%%%%%%%%%%%%%%%
%%%%%%%%%%%%%%%%%%%%%%%%%%%
\section{An application to Banach Space Theory}\label{ApplicationBanach}
%%%%%%%%%%%%%%%%%%%%%%%%%%%
%%%%%%%%%%%%%%%%%%%%%%%%%%%

As already mentioned in the introduction, expressions of the form studied in this work naturally appear in Banach space theory when studying subspaces of the classical Banach space $L_1$. We will use Theorem \ref{thm:generating q of M norm} with $p=2$ to obtain a uniform embedding of the sequence of spaces $\ell^n_q(\ell^n_M)$ ($1<q<2$) into $L_1$, where $M$ is given as in Theorem \ref{thm:generating q of M norm}. In what follows, $r_{ij}$, $i,j=1,\dots,n$ will denote an independent sequence of Rademacher random variables. For any $n\in\N$, the asserted isomorphism is given by
\[
\Psi_n: \ell_{q}^n(\ell_{M}^n) \to L_1 [0,1]^3,\qquad \Psi_n\big((a_{ij})_{i,j=1}^n\big)(t,s,u) = \sum_{i,j=1}^na_{ij}r_{ij}(t)\xi_i(s)X_j(u),
\]
where $r_{ij}$, $\xi_i$, $X_j$, $i,j=1,\dots,n$ are independent collections of i.i.d. random variables defined on the probability space $([0,1],\mathcal B_{[0,1]}, \dif x)$, where $\xi_1$ and $X_1$ have the distributions given by Theorem \ref{thm:generating q of M norm}.
Khintchine's inequality implies
\begin{align*}
\left\| \Psi_n\big((a_{ij})_{i,j=1}^n\big) \right\|_1 & = \int_0^1\int_{0}^1 \int_0^1 \Big| \sum_{i,j=1}^na_{ij}r_{ij}(t) \xi_i(s)X_j (u) \Big| \dif t \dif s \dif u\cr
& \simeq \int_0^1 \int_0^1 \Big(\sum_{i,j=1}^n |a_{ij} \xi_i(s) X_j(u)|^2 \Big)^{1/2} \dif s \dif u \cr 
& = \mathbb E_\xi \mathbb E_X \Big( \sum_{i,j=1}^n |a_{ij}\xi_i X_j|^2\Big)^{1/2}.
\end{align*}
Now, from Theorem \ref{thm:generating q of M norm}, we obtain
\[
\mathbb E_\xi \mathbb E_X \Big( \sum_{i,j=1}^n |a_{ij}\xi_i X_j|^2\Big)^{1/2} \simeq_q \big\| \big( \|(a_{ij})_{i=1}^n \|_{M} \big)_{j=1}^n\big\|_q,
\]
i.e., the sequence of spaces $\left(\ell_{q}^n(\ell_{M}^n)\right)_{n\in\N}$ embeds uniformly into $L_1[0,1]^3$. Since $p=2$, it is necessary that $M'''\leq 0$ for $X$ to have density $f_X$ as given in \eqref{eq:density with M}, which immediately implies that $M$ is $2$-concave (cf., \cite[Lemma 8.1]{ACPP}). To embed these spaces directly into $L_1[0,1]$, we refer the reader to the proof of \cite[Corollary 6.1]{ACPP}. Note that this application is in the same spirit as the main result in \cite{PS}, but with an integral condition instead of pointwise ones (cf., discussion in Section \ref{discussion integral condition}). At this point, it is important to mention the recent paper \cite{Sch} by G. Schechtman on embeddings of spaces $E(F)$ into $L_1$, where he proved that if $E$ and $F$ are spaces with $1$-unconditional bases such that $E$ is $r$-concave and $F$ is $p$-convex for some $1\leq r<p \leq 2$, then the matrix space $E(F)$ embeds into $L_1$. The techniques are different from the ones used in \cite{PS} or here.

%%%%%%%%%%%%%%%%%%%%%%%%%%%%%
%%%%%%%%%%%%%%%%%%%%%%%%%%%%%
\section{The expectation of arbitrary Orlicz norms of random vectors}
%%%%%%%%%%%%%%%%%%%%%%%%%%%%%
%%%%%%%%%%%%%%%%%%%%%%%%%%%%%

One of the key tools we used throughout this paper is Theorem \ref{thm:orlicz_p_norm}. As it turns out, one can rather easily prove a generalization of Theorem \ref{thm:orlicz_p_norm}, therefore providing a simplification of Theorem 1 from \cite{GLSW1}. For the sake of completeness, also in view of the work \cite{ACPP}, we include it here.

\begin{thm}\label{generalization}
Let $N$ be an almost everywhere twice differentiable, normalized Orlicz function such that $N''$ is integrable and $N'(0)=0$. Let $X_1,\dots, X_n$ be independent copies of an integrable random variable $X$. For all $s\geq 0$ define
\[
M(s)= \int_{0}^\infty N(sx) \dif \mathbb P_{|X|}(x).
\]
Then $M$ is a normalized Orlicz function and for all $a=(a_i)_{i=1}^n\in\R^n$,
\[
c_1\|a\|_M \leq \mathbb E \|(a_iX_i)_{i=1}^n\|_N \leq c_2 \|a\|_M,
\]
where $c_1,c_2$ are absolute positive constants.
\end{thm}
%\begin{thm}\label{generalization}
%Let $N$ be a twice differentiable, normalized Orlicz function such that $N''$ is integrable and $N'(0)=0$. Let $X_1,\dots, X_n$ be independent copies of an integrable random variable $X$. For all $s\geq 0$ define
%\[
%M(s)= \int_{0}^\infty N(sx) \dif \mathbb P_{|X|}(x).
%\]
%Then $M$ is an Orlicz function and for all $a=(a_i)_{i=1}^n\in\R^n$,
%\[
%c_1\|a\|_M \leq \mathbb E \|(a_iX_i)_{i=1}^n\|_N \leq c_2 \|a\|_M,
%\]
%where $c_1,c_2$ are absolute positive constants.
%\end{thm}
\begin{proof}
First of all, applying Proposition \ref{PRO_inverse_maximum} to the function $N$, we find independent copies $Y_1,\dots, Y_n$ of a positive integrable random variable $Y$ (which is integrable because $N''$ is), independent of the $X_i$'s, with density
\[
f_Y(y) = y^{-3}N''(y^{-1}),\qquad y>0,
\]
satisfying
\[
\mathbb E_X \mathbb E_Y \max_{1\leq i \leq n} \left| a_i X_i Y_i\right| \simeq \mathbb E_X \| (a_i X_i)_{i=1}^n \|_N.
\]
On the other hand, Theorem \ref{thm:orlicz_carsten1} applied to the sequence of random variables \\$X_1\cdot Y_1,\dots,X_n\cdot Y_n$ (which are integrable due to independence of $X$ and $Y$ gives us a normalized Orlicz function $N_{XY}$, satisfying
\[
\mathbb E_X \mathbb E_Y \max_{1\leq i \leq n} \left| a_i X_i Y_i\right| \simeq \|a\|_{N_{XY}}.
\]
In order to show $N_{XY} = M$, we first observe that
\begin{align*}
N_{XY}(s) & = \int_{0}^s \int_{|XY| \geq t^{-1}} |XY| \dif \mathbb P \, dt \cr
& = \int_{0}^s \int_{|xy| \geq t^{-1}} |xy| \dif \mathbb P_{|X|}(x) \dif \mathbb P_{Y}(y) \dif t \cr
& = \int_{0}^s \int_0^\infty x \int_{(tx)^{-1}}^\infty y  \dif\mathbb P_{Y}(y) \dif \mathbb P_{|X|}(x) \dif t.
\end{align*}
Using the form of the density, a change of variable, and Fubini's theorem, we obtain
\[
N_{XY}(s) = \int_0^s \int_0^\infty xN'(tx) \dif\mathbb P_{|X|}(x)\, \dif t = \int_0^\infty N(sx)\dif\mathbb P_{|X|}(x) = M(s).\qedhere
\]
\end{proof}

\begin{rmk}
Observe that in the case where $N$ is the continuously differentiable, normalized Orlicz function
\[
N(t)=\begin{cases}
t^p, &\text{if }t\leq (p-1)^{-1/p}, \\
p(p-1)^{1/p-1}\cdot t -1, &\text{if }t>(p-1)^{-1/p},
\end{cases}
\]
we have that $\|\cdot\|_N$ is equivalent to $\|\cdot\|_p$, and Theorem \ref{generalization} recovers Theorem \ref{thm:orlicz_p_norm}.
\end{rmk}

\proof[Acknowledgements]
Part of this work was done while J. Prochno visited D. Alonso-Guti\'errez at Universtitat Jaume I in Castell\'on. We would like to thank the department for providing such good environment and working conditions. We are grateful to the anonymous referee for reading this paper so carefully, pointing out mistakes, and providing us with many useful comments that improved the quality of this work.

\bibliographystyle{plain}
\bibliography{tensor_products_random_vectors}

\end{document}